\documentclass[reqno,12pt]{article}

\usepackage{amsmath} 
\usepackage{amssymb}
\usepackage{amsthm}
\usepackage[utf8]{inputenc} 
\usepackage{graphicx} 
\usepackage[english, russian]{babel}
\usepackage{xcolor}
\usepackage[normalem]{ulem}

\numberwithin{equation}{section}

\newtheorem{theo}{Theorem}
\newtheorem{conj}{Conjecture}
\newtheorem{coro}{Corollary}
\newtheorem{prop}{Proposition}

\newtheorem{lem}{Lemma}

\newtheorem{thmx}{Theorem}

\theoremstyle{remark}

\newtheorem*{Remark*}{Remark}
\newtheorem*{Remarks*}{Remarks}

\makeatletter
\newcommand*{\house}[1]{%
 \mathord{%
 \mathpalette\@house{#1}%
 }%
}
\newcommand*{\@house}[2]{%
 \dimen@=\fontdimen8 %
 \ifx#1\scriptscriptstyle\scriptscriptfont
 \else\ifx#1\scriptstyle\scriptfont
 \else\textfont\fi\fi
 3 %
 \sbox0{%
 $#1%
 \vrule width\dimen@\relax
 \overline{%
 \kern2\dimen@
 \begingroup % to keep changes of \dimen@ in #2 local
 #2%
 \endgroup
 \kern2\dimen@
 }%
 \vrule width\dimen@\relax
 \mathsurround=1.5\dimen@ % outside margin
 $%
 }%
 \ht0=\dimexpr\ht0-\dimen@\relax
 \dp0=\dimexpr\dp0+2\dimen@\relax
 \vbox{%
 \kern\dimen@ % reinsert previously removed space
 \copy0 %
 }%
}

\newcommand{\tra}{{}^t}

\newcommand{\Z}{\mathbb{Z}}
\newcommand{\Q}{\mathbb{Q}}
\newcommand{\R}{\mathbb{R}}
\renewcommand{\C}{\mathbb{C}}
\newcommand{\K}{\mathbb{K}}
\renewcommand{\L}{\mathbb{L}}

\newcommand{\Qbar}{\overline{\mathbb Q}}

\newcommand{\eps}{\varepsilon}

\newcommand{\Id}{{\rm Id}}
\newcommand{\Gal}{{\rm Gal}}

\newcommand{\calI}{\mathcal I}
\newcommand{\calE}{\mathcal E}

\newcommand{\calN}{\mathcal N}

\newcommand{\Card}{{\rm Card}}
\newcommand{\nsoul}{\underline{n}}
\newcommand{\isoul}{{\underline{i}}}
\newcommand{\jsoul}{{\underline{j}}}

\newcommand{\GG}{{\bf G}}

\newcommand{\EE}{{\bf E}}

\newcommand{\OK}{{\mathcal O}_{\K}}
\newcommand{\OL}{{\mathcal O}_{\L}}
\newcommand{\OQdexi}{{\mathcal O}_{\Q(\xi)}}
\newcommand{\EK}{\EE_{\K}}
\newcommand{\EQ}{\EE_{\Q}}
\newcommand{\EQbar}{\EE_{\Qbar}}

\begin{document}

\textheight=21cm 
\textwidth=13cm

 \selectlanguage{english}

\title{Values of $E$-functions are not Liouville numbers}
\date\today
\author{S. Fischler and T. Rivoal}
\maketitle

\begin{abstract} Shidlovskii has given a linear independence measure of values of $E$-functions with rational Taylor coefficients at a rational point, not a singularity of the underlying differential system satisfied by these $E$-functions. Recently, Beukers has proved a qualitative linear independence theorem for the values at an algebraic point of $E$-functions with arbitrary algebraic Taylor coefficients. In this paper, we obtain an analogue of Shidlovskii's measure for values of arbitrary $E$-functions at algebraic points. This enables us to solve a long standing problem by proving that the value of an $E$-function at an algebraic point is never a Liouville number. We also prove that values at rational points of $E$-functions with rational Taylor coefficients are linearly independent over $\overline{\mathbb{Q}}$ if and only if they are linearly independent over $\mathbb{Q}$. Our methods rest upon improvements of results obtained by Andr\'e and Beukers in the theory of $E$-operators. 
\end{abstract}

\section{Introduction}
Siegel \cite{siegel} defined in 1929 the class of $E$-functions in order to generalize the Diophantine properties of the exponential function (namely the Lindemann-Weierstrass Theorem) to other special functions such as Bessel's function $J_0(z):=\sum_{n=0}^\infty (-1)^n(z/2)^{2n}/n!^2$ or hypergeometric series ${}_pF_p$ with rational parameters. 
A power series $\sum_{n=0}^\infty \frac{a_n}{n!} z^n \in \Qbar[[z]]$ is said to be an $E$-function when it is solution of a linear differential equation over $\Qbar(z)$ (i.e., holonomic), 
and $\vert \sigma(a_n)\vert$ (for any $\sigma\in \textup{Gal}(\Qbar/\Q)$) and the least common denominator of $a_0, a_1, \ldots, a_n$ all grow at most exponentially in $n$. Note that Siegel's original definition of $E$-functions is more general: see the end of this introduction.
Throughout this paper we fix an embedding of $\Qbar $ in $\C$.

A lot of important qualitative results are known on the arithmetic nature of the values taken by $E$-functions at algebraic points, amongst which we cite the celebrated Siegel-Shidlovskii Theorem (see \cite{Shidlovskiilivre} for instance). This result is not always strong enough, in particular to deduce  linear independence of values of $E$-functions.  It has been improved in \cite{ns} and \cite{beukers} (see Theorem~\ref{thbeukers} below for the linear case), but some assumptions still have to be checked (for instance the functions should not be
evaluated at a singular point).

 Quantitative versions of the  Siegel-Shidlovskii Theorem have been proven, such as the measure of algebraic independence of Lang-Galochkin (see \cite[p.~238, Theorem~5.29 and remarks]{EMS} and \cite[Chapter~11]{Shidlovskiilivre}), but they hold only under certain assumptions on algebraic independence or rationality of the coefficients of the $E$-functions.
In the  linear setting, the main result is the following one, due to Shidlovskii \cite[p.~358, Theorem~1, Eq.~(32)]{Shidlovskiilivre}.
\begin{thmx}[Shidlovskii] \label{thshid}
Let $\underline f = \tra (f_1,\ldots, f_N) \in \mathbb{Q}[[z]]^N$ be a vector of $E$-functions which is solution of a differential system $\underline f '=A\underline f $ for some $A\in M_N(\Q(z))$. Assume that $f_1,\ldots, f_N$ are line\-ar\-ly independent over $\Q(z)$ and that $z_0\in \mathbb Q^*$ is not a pole of an entry of $A$. Then for any $\varepsilon>0$, there exists $c=c(\varepsilon,z_0,f_1,\ldots,f_N)>0$ such that for all $\lambda_1, \ldots, \lambda_N \in \mathbb Z$ not all zero, we have 
$$
\bigg\vert \sum_{j=1}^N \lambda_j f_j(z_0)\bigg\vert > c H^{-N+1-\varepsilon} \quad\mbox{ where } H:=\max_{1\leq j \leq N}\vert\lambda_j\vert.
$$
\end{thmx}
This theorem holds {\em verbatim} with $\mathbb Q$ replaced by an imaginary quadratic number field and $\mathbb Z$ replaced by its ring of integers. However no such result is known for other number fields $\K$. The point is that all known quantitative results are based on the Siegel-Shidlovskii method only, which provides linear independence of the full set of the values of $E$-functions in Theorem~\ref{thshid} only when $\K$ is either $\Q$ or imaginary quadratic. Even the qualitative part of Theorem~\ref{thshid} (namely, $\sum_{j=1}^N \lambda_j f_j(z_0)\neq 0 $) has been proved only recently by Beukers \cite[Corollary 1.4]{beukers} for arbitrary number fields, using Andr\'e's theory of $E$-operators~\cite{andre}.
\begin{thmx}[Beukers] \label{thbeukers}
Let $\underline f = \tra (f_1,\ldots, f_N) \in \Qbar[[z]]^N$ be a vector of $E$-func\-tions which is solution of a differential system $\underline f '=A\underline f $ for some $A\in M_N(\Qbar(z))$. Assume $f_1,\ldots, f_N$ are line\-ar\-ly independent over $\Qbar(z)$ and that $z_0\in \Qbar^*$ is not a pole of an entry of $A$. Then the num\-bers $f_1(z_0),\ldots, f_N(z_0)$ are line\-ar\-ly independent over $\Qbar$.
\end{thmx}

The purpose of this paper is to prove new Diophantine results using this approach of Andr\'e and Beukers. 
Our first main result is the following theorem, where we remove the rationality assumption on the coefficients and also the non-singularity assumption on $z_0$. We recall that for a non-zero algebraic number $\alpha$, its house $\house{\alpha}$ is the maximum of the moduli of $\alpha$ and of all its 
 Galois conjugates over $\mathbb Q$. We also denote by $\OK$ the ring of integers of a number field $\K$. 

\begin{theo} \label{th1}
Let $\K$ be a number field of degree $d$ over $\Q$, 
 $z_0\in\K$, and $\underline f = \tra (f_1,\ldots, f_N)$ be a vector of $E$-functions with coefficients in $\K$ such that $\underline f '=A\underline f $ for some $A\in M_N(\K(z))$. Then for any $\eps>0$, there exists $c=c(\varepsilon,\K,z_0,f_1,\ldots,f_N)>0$ with the following property.

For any $\lambda_1,\ldots, \lambda_N\in\OK$, if $\Lambda := \lambda_1 f_1(z_0) + \ldots + \lambda_N f_N(z_0) $ is non-zero, then 
$$| \Lambda| > c H^{-d N^d+1-\eps} \quad \mbox{ where } H:=\max_{1\leq j \leq N}\house{\lambda_j}.$$
\end{theo}

\begin{Remarks*}
-- Given arbitrary $E$-functions $f_1,\ldots,f_N$ and any $z_0\in\Qbar$, this theorem applies because there exists a number field $\K$ containing $z_0$ and all coefficients of $f_1,\ldots,f_N$, and the family $(f_1,\ldots, f_N)$ can always be enlarged to satisfy a first-order differential system. This shows, without any assumption on $f_1,\ldots,f_N$ and $z_0$, the existence of $\kappa,c>0$ such that $| \Lambda| > c H^{-\kappa}$ provided $ \Lambda\neq 0$.

-- To our knowledge, Theorem~\ref{th1} provides the first quantitative version of Beukers' Theorem~\ref{thbeukers}, when it is further assumed in Theorem~\ref{th1} that $f_1, \ldots, f_N$ are linearly independent over $\mathbb K(z)$ and that $z_0\in \mathbb K^*$ is not a pole of $A$, ensuring that $\Lambda\neq 0$.

-- The  constants $c$ in Theorems~\ref{thshid} and~\ref{th1}, and their equivalents in other results of the paper, are in principle effective by a remark of Andr\'e \cite[footnote on p.~129]{andrelivre} and by the multiplicity  estimate of Bertrand-Beukers~\cite{bb} made completely explicit by 
Bertrand-Chirskii-Yebbou~\cite{bcy}. However, we did not try to compute them because they would certainly  be far from best possible and because we do not see any immediate application.
\end{Remarks*}

An important consequence of Theorem~\ref{th1} is the following result, which completely settles the problem of deciding whether (real) values of $E$-func\-tions can be Liouville numbers or not. We recall that a Liouville number is a real number $\xi$ such that there exist two sequences of rational integers $p_n, q_n$ such that $q_n\ge 2$ and $0<\vert q_n\xi-p_n\vert<1/q_n^{ n}$ for all sufficiently large integers $n$ ({\em a fortiori} $p_nq_n\neq 0$). 

\begin{coro} \label{corliouville}
Let $f$ be an $E$-function, and $z_0$ be an algebraic number. Then $f(z_0)$ is not a Liouville number. 
\end{coro}
 The proof of Corollary \ref{corliouville} runs as follows: in Theorem~\ref{th1}, let $z_0\in \Qbar$, take $f_1:=1$, $f_2:=f$ and consider a vector $\tra (f_1, f_2\ldots, f_N)$ of $E$-functions with coefficients in a number field $\K$ such that $\underline f '=A\underline f$ for some $A\in M_N(\K(z))$, where $\K$ is large enough to contain $z_0$. If $f(z_0)\in \mathbb Q$, then $f(z_0)$ is not a Liouville number. 
If $f(z_0)\notin \mathbb Q$, then $\lambda_1+\lambda_2 f(z_0)\neq 0$ for all $\lambda_1, \lambda_2\in\Z$ not both zero, so that Theorem~\ref{th1} yields 
$\vert \lambda_1+\lambda_2 f(z_0)\vert>c\max(\vert \lambda_1\vert, \vert \lambda_2\vert)^{-\kappa}$ for some $c,\kappa>0$. This rules out the possibility that $f(z_0)$ is a Liouville number.

Of course Corollary \ref{corliouville} is interesting only when $f(z_0)\in\R$. If we do not assume this, note however that the real and imaginary parts of $f(z_0)$ are values of $E$-functions (see the remark in \S\ref{ssec:conjugates} below), so that none of them is a Liouville number.

 Amongst other families of classical real numbers already proved not to be Liouville numbers, we find automatic numbers (Adamczewski-Cassaigne~\cite{ac}) and Mahler numbers, i.e. the values of Mahler series in $\mathbb Q[[z]]$   at the inverse of rational integers (Bell-Bugeaud-Coons~\cite{bbc}). It is interesting to observe that the theories of $E$-functions and of Mahler series share many common properties, see also \cite{AF}.

\medskip

Let us also mention another interesting corollary, which is a consequence of Theorem~\ref{th1} with $N=2$ and $N=3$ respectively (recall that $J_0(z)$ is solution of $zy''(z)+y'(z)+zy(z)=0$).
\begin{coro} \label{coro:mesexpj0}
For any algebraic number $\alpha\in \Qbar^*$ of degree $d$ over $\mathbb Q$ and any $\varepsilon>0$, there exists $c=c(\alpha,\varepsilon)$ such that, for all $(p,q)\in \Z\times \mathbb N$ with $q\neq 0$, 
$$
\Big \vert e^{\alpha}-\frac pq\Big\vert > \frac{c}{q^{d2^d+\varepsilon}},
\quad \textup{ 
respectively} \quad 
\Big \vert J_0(\alpha)-\frac pq\Big\vert >\frac{c}{q^{d3^d+\varepsilon}}.
$$
\end{coro}
The above mentioned general transcendence measure for $E$-functions due to Lang and Galochkin \cite[p.~238, Theorem~5.29 and remarks]{EMS}, 
refined by Kappe~\cite{kappe} for $e^{\alpha}$ in the special case of irrationality measures,  gives $4d^2-2d$ instead of $d2^d$, and $16d^3+1$ instead of $d3^d$; see also \cite[p.~404]{Shidlovskiilivre}. This is of course much better than Corollary~\ref{coro:mesexpj0} for large $d$ but our bounds turn out to be smaller for $d\in \{2,3\}$ and $d\in \{2,3,4,5\}$ respectively. Note that Zudilin \cite{zudilin} has obtained the optimal exponent 2 for $J_0(\alpha)$ when $\alpha \in \mathbb{Q}^*$.

 A less classical example is the following: for any $\alpha\in \Qbar^*$ and any integers $p,q\ge 1$, the value at $z=\alpha$ of  
 $\mathcal{A}_{p,q}(z):=\sum_{n=0}^\infty (\sum_{k=0}^n \binom{n}{k}^p\binom{n+k}{n}^q)z^n/n!$ is not a Liouville number; $\mathcal{A}_{p,q}(\alpha)$ is proved to be a transcendental number in \cite[\S 4.6]{BRS2} when $(p,q)\in \{1,2,3,4\}^2$, the situation in general being unknown. More specifically, it is also proved in \cite[\S 4.6, Table~1]{BRS2} that when $(p,q)=(2,2)$, the minimal inhomogeneous differential equation over $\mathbb Q(z)$ satisfied by $\mathcal{A}_{2,2}$ is of order $4$ and 0 is its only singularity. Hence, the following holds by Theorem~\ref{th1} with $N=5$: for all $\alpha\in \Qbar^*$ of degree $d$ over $\mathbb Q$ and all $\varepsilon>0$, there exists $c=c(\varepsilon, \alpha)>0$ such that for all $\lambda_j\in \mathbb Z$ not all 0, we have
$$
\Big\vert \lambda_4+\sum_{j=0}^3 \lambda_j \mathcal{A}_{2,2}^{(j)}(\alpha) \Big\vert >c 
\max_{0\le j\le 4}\vert \lambda_j\vert^{ -d5^d+1-\varepsilon}. 
$$
In particular, for any $\alpha\in \Qbar^*$ of degree $d$ and $\eps>0$, there exists $c=c(\varepsilon, \alpha)>0$ such that for any $(p,q)\in\Z\times \mathbb N$ with $q\neq0$, 
\begin{equation}\label{eq:A22}
\Big\vert \mathcal{A}_{2,2}(\alpha) -\frac{p}{q}\Big\vert >\frac{c}{q^{
d5^d+\varepsilon}}.
\end{equation}

It seems reasonable to make the following conjecture, which probably belongs to folklore. 
This conjecture is Roth's Theorem if $f(z_0)$ is algebraic.
\begin{conj} \label{conjmu2}
Let $f$ be an $E$-function and $z_0\in\Qbar$. For any $\eps>0$, there exists $c>0$ such that for any $(p,q)\in\Z\times \mathbb N$ with $q\neq 0$, either $q f(z_0)-p=0 $ or 
$$ \Big| f(z_0)-\frac p q \Big| > \frac c {q^{2+\eps}}.$$
\end{conj}
Zudilin has proved this conjecture in \cite{zudilin} in a stronger form but under additional assumptions (which even imply $f(z_0)\notin \mathbb Q$), namely:~$f$ is an $E$-function with rational coefficients, $z_0\in \mathbb Q^*$ is not a singularity of a differential system satisfied by~$1, f, f_2, \ldots, f_N$ over $\mathbb Q(z)$ for some $N\ge 2$ and $f, f_2, \ldots, f_N$ are algebraically independent over $\mathbb Q(z)$. It would be interesting to know if $\mathcal{A}_{2,2}, \mathcal{A}'_{2,2}, \mathcal{A}''_{2,2}, \mathcal{A}'''_{2,2}$ are algebraically independent over $\mathbb Q(z)$, in which case the exponent $5$ could be improved to 2 in \eqref{eq:A22} when $\alpha\in \mathbb Q^*$.

\bigskip

With the notation of  Theorem \ref{th1}, for any $D\geq 1$ one may consider the $\binom{N+D-1}{N-1 }$ functions $f_1^{i_1}\ldots f_N^{i_N}$, with non-negative integers $i_1$, \ldots, $i_N$ such that $i_1+\ldots+i_N=D$. They make up   a vector solution of a first order differential system, to which Theorem \ref{th1} applies. Taking $f_1=1$ and $f_2=f$, this provides the following transcendence measure, valid for any value of an $E$-function at an algebraic point.

\begin{coro}  \label{corotrmes} Let $f$ be an $E$-function, with coefficients in a number field $\mathbb{K}$ of degree $d$, solution of an inhomogeneous linear differential equation of order $N-1$. Let $z_0 \in \mathbb K$. Then for any  $D\geq 1$, any $\eps>0$, and  any $a_0,\ldots,a_D\in\Z$,   either $\sum_{j=0}^D a_j f(z_0)^j=0$ or
$$ \bigg| \sum_{j=0}^D a_j f(z_0)^j \bigg| \geq c \, \big( \max( |a_0|,\ldots, |a_D|)\big)^{-\kappa-\eps} $$ 
where $c$ depends only on $f$, $z_0$, $D$ and $\eps$, and 
$\kappa = d \binom{N+D-1}{N-1} ^d -1$.  
\end{coro}
With our notations, the transcendence measure of Lang-Galochkin provides the exponent $\frac{2^{N}(N-1)^{N-1}}{(N-1)!}d^{N}D^{N-1}$: for fixed $N$, it is smaller than $\kappa$ when $d$ and $D$ are large but it is sometimes larger for small values of $d$ or $D$ as mentioned after Corollary \ref{coro:mesexpj0}, and moreover it holds under an algebraic independence assumption on the $E$-functions in the differential system. No such assumption is necessary in Corollary~\ref{corotrmes}, which adds value to this result. Both results can be applied when   $f$ is the exponential function (with $N=2$): we obtain a transcendence measure of $e^{z_0}$ for any $z_0\in \Qbar^*$ of degree $d$ with $\kappa = d(D+1)^d-1$  whereas Lang-Galochkin's result yields $4d^2D$. Note that when $d=1$ and $N=2$ as for the exponential, we obtain $\kappa=D$ which is best possible. However this is not a new result for the exponential at rational points (see  
\cite[Chapter 10]{baker}, and  \cite[p.~135, Satz~3]{mahler} when $z_0=1$).

Corollary~\ref{corotrmes} shows that no value of an $E$-function at an algebraic point can be a $U$-number, in Mahler's classification (see for instance \cite[Chapter~3]{bugeaud}). It seems reasonable to conjecture that it can neither be a $T$-number, so that values of $E$-functions at algebraic points would be either algebraic numbers or $S$-numbers.

 \medskip

The second goal of this paper is to understand the structure of the ring $\EE $ of all values at algebraic points of $E$-functions; this algebraic point can always be assumed to be $1$ because if $f(z)$ is an $E$-function, so is $f(\alpha z)$ for any $\alpha\in \Qbar$. Elements of $\EE$ are related to exponential periods (see \cite[\S 4.3]{fuchsien}). 

For any subfield $\mathbb K$ of $\Qbar$, we shall also consider the subring $\EK$ of $\EE$ which consists of the evaluations $f(1)$ where $f$ is an $E$-function with coefficients in $\mathbb K$ (the number 1 could be replaced by any non-zero element of $\mathbb K$ without changing $\EK$).
Note that $\EE$ is the union of all $\EK$, where $\K$ is a number field, since for any $E$-function $f$ the holonomy property implies the existence of a number field that contains all coefficients of $f$.

We have defined and studied \cite{gvalues} analogous rings $\GG_\K$ with $G$-functions instead of $E$-functions; it turns out that $\GG_\K$ is nearly independent from $\K$ (precisely, $\GG_\K= \GG_\Q$ if $\K\subset\R$, and $\GG_\K=\GG_{\Q(i)}$ otherwise). The situation is completely different for $E$-functions. A first hint in this direction was given in \cite[Theorem 4]{ateo} (stated as Lemma~\ref{lemateo} in \S \ref{secdcp} below). The way $\EK$ depends on $\K$ is completely described by the following result.

\begin{theo} \label{theotens} Let $\K$ be a subfield of $\Qbar$. 
Elements of $\EK$ are linearly independent over $\Qbar$ if, and only if, they are linearly independent over $\K$.
\end{theo}

We refer to \cite[Theorem 1.7]{AF} for a similar result concerning Mahler functions. 
Theorem~\ref{theotens} means that the $\K$-algebras $\EK$ and $\Qbar$ are linearly disjoint, and the natural map $\EK\otimes_\K \Qbar \to \EE $ (sending $\xi\otimes z$ to $\xi z$) is a $\K$-algebra isomorphism.
The definition and properties of linearly disjoint algebras can be found in \cite[Chapter V, \S 2, No. 5]{Bourbaki}; using Theorem~\ref{theotens} with $\K=\Q$ they imply the following.

\begin{coro}\label{cordcp}
Let $\K$ be a number field, and $(\omega_1,\ldots,\omega_d)$ be a basis of the $\Q$-vector space $\K$. Then 
$$\EK = \omega_1 \EQ\oplus \ldots \oplus \omega_d \EQ.$$
In other words, for any $\xi\in\EK$ 
there exists a unique $d$-tuple $(\xi_1,\ldots,\xi_d)\in\EQ^d$ such that $\xi = \omega_1\xi_1+\ldots+\omega_d\xi_d$.
\end{coro}

\medskip

The original definition of $E$-functions, given by Siegel \cite{siegel}, is slightly less restrictive:
instead of geometric bounds, 
he allowed growths bounded by $n!^{\varepsilon}$ (for any given $\varepsilon>0$, provided $n$ is large enough with respect to $\varepsilon$). 
Shidlovskii's Theorem~\ref{thshid} holds for $E$-functions in Siegel's sense, and
 Beukers' Theorem~\ref{thbeukers} was later proved by Andr\'e \cite{andreasens} in this general setting, by a different method. All other results in Beukers' paper \cite{beukers} have been adapted by Lepetit \cite{lepetit}. Therefore all the results of the present paper also hold for $E$-functions in Siegel's sense.

\medskip

The structure of this paper is as follows. In \S \ref{sec2} we prove Proposition~\ref{prop1}, 
which is crucial in the proof of both Theorems \ref{th1} and \ref{theotens}. We also derive from it 
 a version of Beukers' desingularization process over a number field. This enables us to prove Theorem~\ref{th1} in \S \ref{secpreuveth1}, and also to obtain in \S \ref{secdcp} a decomposition of an $E$-function over a number field, involving an $E$-function that takes only transcendental values at non-zero algebraic points. At last we apply the previous results in \S \ref{sec5} to study the structure of $\EK$ and prove Theorem~\ref{theotens}. 
The final section is devoted to an action of $\Gal(\Qbar/\Q)$ on the set of values of $E$-functions; it is not used in the paper but it illustrates how our method presents similarities with Liouville's theorem.

\bigskip

\noindent {\bf Acknowledgments.} We warmly thank the referees for their very contructive comments.

\section{Main tools}\label{sec2}

\subsection{Conjugates of $E$-functions} \label{ssec:conjugates}

Let $f(z) = \sum_{n=0}^\infty a_n z^n/n!$ be an $E$-function with coefficients $a_n\in\Qbar$. For any $\sigma \in\Gal(\Qbar/\Q)$ we let $f^\sigma(z) = \sum_{n=0}^\infty \sigma(a_n) z^n/n!$. The definition shows that  $f^\sigma$ is also an $E$-function, and if $g$ is an $E$-function then for any $\sigma ,\tau$ we have 
$(f+g)^\sigma = f^\sigma+g^\sigma$, $ (f g)^\sigma = f^\sigma g^\sigma$ and $(f^\sigma )^\tau = f^{\tau\circ\sigma}$. Moreover if $f$ has coefficients in a number field $\K$, then $f^\sigma$ has coefficients in the number field $\sigma(\K)$.

\medskip

\begin{Remark*} 
Denoting by $\sigma $ the complex conjugation, for any $E$-function $f$ we can consider $\frac12(f+f^\sigma)$ and $\frac1{2i} (f-f^\sigma)$. These $E$-functions have real coefficients, which are respectively the real and imaginary parts of those of $f$. In particular, the real and imaginary parts of any element of $\EE$ belong to $\EE$.\end{Remark*}

\medskip

The following result is central in the present paper; we refer to \cite[Proposition 3.5]{BRS2} for a similar result. We agree that minimal polynomials of algebraic elements have leading coefficient 1.

\begin{prop}\label{prop1} Let $f$ be an $E$-function with coefficients in a number field $\K$, and $z_0\in \Qbar^\ast$. Then the following assertions are equivalent:
\begin{itemize}
\item[$(i)$] $f$ vanishes at $ z_0 $.
\item[$(ii)$] There exists $\sigma \in\Gal(\Qbar/\Q)$ such that $f^\sigma $ vanishes at $ \sigma(z_0) $.
\item[$(iii)$] For any $\sigma \in\Gal(\Qbar/\Q)$, $f^\sigma $ vanishes at $ \sigma(z_0) $.
\item[$(iv)$] There exists an $E$-function $g$ with coefficients in $\K$ such that 
$$f(z) = D(z) g(z) \mbox{ where $D$ is the minimal polynomial of $z_0$ over $\K$.}$$
\end{itemize}
\end{prop}

In particular, if $z_0$ is rational and $f$ vanishes at $z_0$, then all conjugates $f^\sigma$ of $f$ also vanish at $z_0$.
Also, if an $E$-function $f$ with rational coefficients vanishes at some $z_0\in \Qbar^\ast$, then it vanishes at all Galois conjugates of $z_0$.

\medskip

We remark that with $z_0=1$, the implication $(i)\Rightarrow (iii)$ is used already in the proof of \cite[Proposition~4.1]{beukers}, which is the main result Proposition~\ref{prop1} is based on.

\medskip

\begin{proof}[Proof of Proposition~\ref{prop1}] $(iv)\Rightarrow (iii)$ Let $\sigma \in\Gal(\Qbar/\Q)$. Then $f^\sigma(z) =D^\sigma(z) g^\sigma(z) $, and $D^\sigma(\sigma(z_0))=\sigma(D(z_0))=0$. Therefore $f^\sigma(\sigma(z_0))=0$.

\smallskip

$(iii)\Rightarrow (ii)$ is trivial.

\smallskip

$(ii)\Rightarrow (i)$ Enlarging $\K$ if necessary, we may assume the extension $\K/\Q$ to be Galois and to contain $z_0$. Then $f^\sigma$ has coefficients in $\K$, and $\sigma(z_0)\in\K^\ast$. Using \cite[Proposition~4.1]{beukers} there exists an $E$-function $g$ such that $f^\sigma(z) = (z-\sigma(z_0))g(z)$. Then $g$ has coefficients in $\K$; applying $\sigma^{-1}$ yields $f(z) = (z-z_0)g^{\sigma^{-1}}(z)$ so that $f(z_0)=0$. 

\smallskip

$(i)\Rightarrow (iv)$ Using \cite[Proposition~4.1]{beukers} there exists an $E$-function $h$ such that $f (z) = (z- z_0 )h(z)$. Let $\sigma \in\Gal(\Qbar/\K)$, that is: $\sigma$ is a field automorphism of $\Qbar$ such that $\sigma(x)=x $ for any $x\in \K$. Then $f(z) = f^\sigma(z) = (z-\sigma(z_0))h^\sigma(z)$ so that $f$ vanishes at $\sigma(z_0)$. Let $z_1:=z_0$, $z_2$, \ldots, $z_\ell$ denote the (pairwise distinct) Galois conjugates of $z_0$ over $\K$, i.e. the elements of the form $\sigma(z_0)$ with $\sigma \in\Gal(\Qbar/\K)$; we have proved that $f$ vanishes at $z_1$, \ldots, $z_\ell$. Applying \cite[Proposition~4.1]{beukers} yields, by induction on $j\in\{1,\ldots, \ell\}$, the existence of an $E$-function $g_j$ such that $f (z) = g_j(z) \prod_{i=1}^j (z-z_i)$. Since $D(z) = \prod_{i=1}^\ell (z-z_i)$, we have $f(z) = D(z) g_\ell(z)$. Now $D(z)\in\K[z]\setminus\{0\}$ so that all coefficients of $g_\ell$ belong to $\K$. This concludes the proof of Proposition~\ref{prop1}.
\end{proof}

\subsection{Beukers' desingularization process}

In the proof of Theorem~\ref{th1} we shall use the following version of Beukers' desingularization theorem  \cite[Theorem 1.5]{beukers}.

\begin{prop}\label{prop2} 
Let $\K$ be a number field, and $f_1,\ldots,f_N$ be $E$-functions with coefficients in $\K$, linearly independent over $\C(z)$. Assume that the vector $\underline f = \tra (f_1,\ldots,f_N)$ satisfies a first-order differential system $\underline f '=A\underline f $ with $A\in M_N(\K(z))$. 

Then there exist $E$-functions $e_1,\ldots,e_N$ with coefficients in $\K$, linearly independent over $\C(z)$, a matrix $B\in M_N(\K[z,1/z])$ and a matrix $M\in M_N(\K[z ])$, such that with $\underline e = \tra (e_1,\ldots,e_N)$:
$$ \underline e '=B\underline e\quad \mbox{ and }\quad \underline f = M \underline e.$$
\end{prop}

 The new point is that $e_1,\ldots,e_N$ and the coefficients of $B$ and $M$ have coefficients in the number field $\K$ (whereas in \cite[Theorem 1.5]{beukers} these   are simply algebraic numbers).
 
\medskip

The proof of Proposition~\ref{prop2}  follows~\cite[p.~378]{beukers}, using also the additional details given in \cite{BRS2}. Actually Proposition~\ref{prop2} is already proved
implicitly (for $\K=\Q$) by the implementation described in~\cite{BRS2}. 

\medskip

In what follows we  only mention the parts of the proof where   special attention has to be paid. Let $\alpha$ be a singularity of the differential system $Y'=AY$, and $Q\in\K[X]$ denote the minimal polynomial of $\alpha$ over $\K$. Let $k\geq 1$ be the maximal order of $\alpha$ as a pole of a coefficient of $A$, and $(i_0,j_0)$ be such that $A_{i_0,j_0}$ has a pole of order exactly $k$ at $\alpha$. Then $Q^k A\underline f = Q^k \underline f'$ vanishes at $\alpha$; the $i_0$-th coordinate of this vector provides a linear relation
 $$ \sum_{j=1}^N (Q^k A_{i_0,j})(\alpha) f_j(\alpha)=0.$$
 Note that for any $j$, the rational function $Q(z)^k A_{i_0,j}(z)\in\K(z)$ is holomorphic at $\alpha$, and for $j=j_0$ it does not vanish at that point. Multiplying by the least common denominator of the rational functions $Q(z)^k A_{i_0,j}(z) $    we obtain coprime polynomials $P_1,\ldots,P_N\in\K[z]$ such that 
 $$P_{j_0}(\alpha)\neq 0 \quad \mbox{ and }\quad \sum_{j=1}^N P_j(\alpha) f_j(\alpha)=0.$$
  If $N=1$ we let $P_{1,1}=1$; otherwise there exist polynomials $P_{i,j}\in\K[z]$, for $2\leq i \leq N$ and $1\leq j \leq N$, such that letting $P_{1,j}=P_j$, the matrix $S = (P_{i,j})_{1\leq i,j\leq N}\in M_N(\K[z])$ has determinant 1. Then $S \underline f $ is a vector of $E$-functions, with coefficients in $\K$, of which the first coordinate $ \sum_{j=1}^N P_j(z) f_j(z)$ vanishes at $ \alpha $. Using Proposition~\ref{prop1}, we deduce that $ \sum_{j=1}^N P_j(z) f_j(z)$ vanishes at $\sigma(\alpha)$ for any $\sigma\in\Gal(\Qbar/\K)$.  This concludes the proof as in \cite{beukers}.

\section{Proof of Theorem~\ref{th1}} \label{secpreuveth1}

To prove Theorem~\ref{th1} we assume $z_0\neq 0$ since otherwise,  $f_1(z_0),\ldots,f_N(z_0)$ are algebraic numbers and  the conclusion follows from Schmidt's subspace theorem (see for instance \cite[Chapter 1, \S 8.2, Theorem~1.37]{EMS}). Considering the $E$-functions $f_j(z_0z)$ instead of $f_j(z)$, we may even assume $z_0=1$. We also suppose that $\Lambda\neq 0$. 

The proof of Theorem~\ref{th1} falls into 2 steps.

\bigskip

\noindent {\bf Step 1}. Let us prove Theorem~\ref{th1} in the special case where $\K=\Q$ (i.e., $d=1$). In other words, we assume that $z_0=1\in\Q$, $\lambda_1,\ldots,\lambda_N\in\Z$, and $f_1,\ldots,f_N$ have coefficients in $\Q$.

Recall that we do not assume that $f_1(1),\ldots,f_N(1)$ are linearly independent over $\Q$. Let $N'$ be the maximal number of linearly independent numbers among them. We may assume (up to a permutation of the indices) that $f_1(1),\ldots,f_{N'}(1)$ are linearly independent over $\Q$, and $f_{N'+1}(1),\ldots,f_{N}(1)$ belong to the $\Q$-vector space they span. There exist rational numbers $\varrho_{i,j} $ such that
 $f_j(1) = \sum_{i=1}^{N'} \varrho_{i,j} f_i(1)$ for any $1\leq j \leq N$, so that 
\begin{equation}\label{eqlam2}
\Lambda := \sum_{j=1}^N \lambda_j f_j(1) = \sum_{i=1}^{N'} \mu_i f_i(1) \quad \mbox{ with} \quad 
\mu_i := \sum_{j=1}^N \lambda_j \varrho_{i,j}\in\Q.
\end{equation}

Observe that the $E$-functions $f_1,\ldots,f_{N'}$ are linearly independent over $\C(z)$. Indeed, otherwise they would be linearly dependent over $\Q(z)$ (since they have coefficients in $\Q$), and a relation $\sum_{j=1}^{N'} S_j(z) f_j(z)=0$ would exist with $S_1, \ldots, S_{N'}\in \Q(z)$ not all zero. Upon multiplying by $(z-1)^k$ for a suitable $k\in\Z$, we may assume that none of the $S_j$ has a pole at 1, and that at least one of them does not vanish at 1. This provides a non-trivial linear relation $\sum_{j=1}^{N'} S_j(1) f_j(1)=0$, which contradicts the definition of $N'$. 

Therefore $f_1,\ldots,f_{N'}$ are linearly independent over $\C(z)$. Denote by $N''$ the dimension of the vector space generated over $\C(z)$ by $f_1,\ldots,f_{N}$; we have $N'\leq N''\leq N$. Notice that it could happen that $N''>N'$, for instance if $f_{N'+1}(1)=0$. Up to a permutation of the indices, we may assume that $f_1,\ldots,f_{N''}$ are linearly independent over $\C(z)$, and that 
$f_{N''+1},\ldots,f_{N}$ belong to the vector space they span over $\C(z)$.

Since $f_1,\ldots,f_{N''}$ are linearly independent over $\C(z)$, and satisfy a linear differential system of order 1 by definition of $N''$, Proposition~\ref{prop2} (applied with $\mathbb K=\mathbb Q$) provides $E$-functions $e_1$, \ldots, $e_{N''}$ with rational coefficients and matrices $B\in M_{N''}(\Q[z,1/z])$ and $M = (P_{i,j})\in M_{N''}(\Q[z])$  such that $\underline e' = B \underline e$ and $f_i(z) = \sum_{j=1}^{N''} P_{i,j}(z)e_j(z)$. Since $N'\leq N''$, Eq.~\eqref{eqlam2} yields
\begin{equation}\label{eqlam3}
\Lambda = \sum_{j=1}^{N''} \nu_j e_j(1) \quad \mbox{with} \quad 
\nu_j := \sum_{i=1}^{N'} \mu_i P_{i,j}(1)\in\Q.
\end{equation}
Now Shidlovskii's lower bound stated as Theorem~\ref{thshid} in the introduction 
applies to the $E$-functions $e_1$, \ldots, $e_{N''}$ with rational coefficients, which are linearly independent over $\C(z)$ and solution of a linear differential system of order 1 of which 1 is not a singularity. Denoting by $\delta$ a common denominator of the rational numbers $P_{i,j}(1)$ and $\varrho_{i,j} $ (appearing in Eqs.~\eqref{eqlam2} and \eqref{eqlam3}), we obtain that $\delta^2 \Lambda $ is a $\Z$-linear combination of $e_1(1)$, \ldots, $e_{N''}(1)$ with coefficients bounded (in absolue value) by $cH$, where $c>0$ and $\delta$ depend only on $f_1$, \ldots, $f_N$. For any $\eps>0$, Theorem~\ref{thshid} yields $ |\delta^2 \Lambda |>c_0 H^{-N''+1-\eps}\geq c_0 H^{-N+1-\eps}$ for some $c_0>0$ which depends only on $f_1$, \ldots, $f_N$ and $\eps$. This concludes the proof of Theorem~\ref{th1} in the case where $\K=\Q$.

\medskip

\noindent {\bf Step 2}. Let us prove Theorem~\ref{th1} for any number field $\K$. For simplicity of the exposition, we assume $\K$ to be a Galois extension of $\Q$; see the end of Step 2 for the general case. We denote by $G$ the Galois group of $\K/\Q$, and consider the complex number 
\begin{equation} \label{eqdefvarpi}
\varpi := \prod_{\sigma\in G} \Big( \sum_{j=1}^N \sigma(\lambda_j) f_j^\sigma(1)\Big).
\end{equation}

To begin with, let us prove that $ \varpi \neq 0$. Indeed, consider the $E$-function $g(z) = \sum_{j=1}^N \lambda_j f_j(z)$; it has coefficients in $\K$ (because $f_1$, \ldots, $f_N$ do), and $g(1) = \Lambda\neq 0$. For any $\sigma \in G$, Proposition~\ref{prop1} yields $g^\sigma(1) \neq 0$. Now $g^\sigma(1) = \sum_{j=1}^N \sigma(\lambda_j) f_j^\sigma(1)$, so that $\varpi = \prod_{\sigma\in G} g^\sigma(1) \neq 0$.

Denote by $\calI   $ the set of all tuples of integers $\isoul = (i_\sigma)_{\sigma\in G}$ such that $1 \leq i_\sigma \leq N$ for any $\sigma\in G$. Expanding the product in the definition \eqref{eqdefvarpi} of $\varpi$ yields
\begin{equation} \label{eqdefvp}
\varpi = \sum_{\isoul\in\calI} 
\prod_{\sigma\in G} \sigma(\lambda_{i_\sigma}) f_{i_\sigma}^\sigma(1)
= \sum_{\isoul\in\calI} 
\Big( \prod_{\sigma\in G} \sigma(\lambda_{i_\sigma})\Big) g_\isoul(1)
\end{equation}
where
\begin{equation} \label{eqdefgi}
g_\isoul(z) := \prod_{\sigma\in G} f_{i_\sigma}^\sigma(z) 
\end{equation}
is an $E$-function with coefficients in $\K$. 

The normal basis theorem (see for instance \cite[Theorem 5.18]{Milne}) provides an element $\alpha\in\K$ such that the $\sigma(\alpha)$, for $\sigma\in G$, make up a basis of the $\Q$-vector space $\K$. Upon multiplying $\alpha$ by a suitable positive integer, we may assume that $\alpha\in\OK$ (so that $\sigma(\alpha)\in\OK$ for any $\sigma\in G$). Expanding all coefficients of $ g_\isoul(z)$ in this basis yields (using \cite[Chapter 3, Lemma~12]{Shidlovskiilivre}) $E$-functions $ g_{\isoul,\sigma}(z)$ with coefficients in $\Q$, for $\sigma\in G$, such that 
\begin{equation} \label{eqdcpgi}
 g_\isoul(z) = \sum_{\sigma\in G} \sigma(\alpha)g_{\isoul,\sigma}(z)
 \quad \mbox{ for any } \isoul\in\calI.
\end{equation}
In the sequel, it is important to observe that these $E$-functions $ g_{\isoul,\sigma}(z)$ are uniquely determined by $ g_\isoul(z)$, since for each $n$ their coefficients of $z^n$ are given by the expansion in the basis $(\sigma(\alpha))_{\sigma\in G}$ of the corresponding coefficient of $ g_\isoul(z)$.

For $ \isoul\in\calI$ and $\sigma\in G$, we denote by $\sigma(\isoul)$ the tuple $\jsoul\in\calI$ defined by $j_\tau = i_{\sigma\circ\tau}$ for any $\tau\in G$. Let us prove that 
\begin{equation} \label{eqegalg}
g_{\isoul,\sigma}(z) = g_{\sigma(\isoul),\Id}(z) \quad \mbox{ for any } \isoul\in\calI
 \mbox{ and any } \sigma\in G.
\end{equation}
Indeed we have:
\begin{align*}
\sum_{\tau\in G} \tau(\alpha) g_{\sigma(\isoul),\tau}(z) 
&=
g_{\sigma(\isoul)}(z) \mbox{ using Eq.~\eqref{eqdcpgi}}
\\
&=
\prod_{\tau\in G} f_{i_{\sigma\circ\tau}}^\tau(z)
\mbox{ by definition of $g_{\sigma(\isoul)}(z)$}
\\
&=
\prod_{\tau'\in G} f_{i_{ \tau'}}^{\sigma^{-1}\circ \tau'}(z) \quad \mbox{ by letting }
\tau' = \sigma\circ\tau\\
&=
\Big( \prod_{\tau'\in G} f_{i_{ \tau'}}^{ \tau'}\Big)^{\sigma^{-1}}(z)\\
&= g_\isoul^{\sigma^{-1}}(z) \\
&=
\sum_{\tau\in G} \sigma^{-1}(\tau(\alpha)) g_{ \isoul ,\tau}(z) 
 \quad \mbox{ since $g_{ \isoul ,\tau} $ has coefficients in $\Q$.}
\end{align*}
Comparing the coefficient of $\alpha$ on both sides yields Eq.~\eqref{eqegalg} since an expansion for $g_{\sigma(\isoul)}(z)$ like the one of Eq.~\eqref{eqdcpgi} is unique.

Let us prove now that the vector of $E$-functions $g_{ \isoul ,\Id}(z)$, for $\isoul\in\calI$, is solution of a first-order linear differential system. By assumption we have $f_i'(z) = \sum_{j=1}^N A_{i,j}(z) f_j(z)$ with $A_{i,j}\in\K(z)$, so that 
$ (f_i^\sigma)'(z) = (f_i')^\sigma(z)=\sum_{j=1}^N \sigma(A_{i,j})(z) f_j^\sigma(z)$ and
\begin{align*}
g_\isoul '(z) 
&=\sum_{\sigma\in G} ( f_{i_\sigma}^\sigma)'(z) \prod_{\tau\neq \sigma} 
 f_{i_\tau}^\tau (z)\\
 &= \sum_{\sigma\in G} \sum_{j=1}^N \sigma(A_{i_\sigma,j})(z) f_j^\sigma(z) 
 \prod_{\tau\neq \sigma} f_{i_\tau}^\tau(z) \\
&= \sum_{\isoul' \in\calI } B_{\isoul,\isoul'}(z)g_{\isoul'}(z) \quad \mbox{ for some }
B_{\isoul,\isoul'}(z)\in\K(z)\\
&= \sum_{\isoul' \in\calI } B_{\isoul,\isoul'}(z) \sum_{\sigma\in G} \sigma(\alpha) g_{\sigma(\isoul'),\Id}(z) \quad \mbox{ using Eqs.~\eqref{eqdcpgi} and \eqref{eqegalg}}\\
&= \sum_{\isoul'' \in\calI } C_{\isoul,\isoul''}(z) g_{\isoul'',\Id}(z) \quad \mbox{ for some }
C_{\isoul,\isoul''}(z)\in\K(z).
\end{align*}
Each $C_{\isoul,\isoul''}(z)\in\K(z)$ can be written as $N_{\isoul,\isoul''}(z)/D_{\isoul,\isoul''}(z)$ with $D_{\isoul,\isoul''}(z)\in\Q[z]\setminus\{0\}$ and $N_{\isoul,\isoul''}(z)\in\K[z]$. Writing $N_{\isoul,\isoul''}(z)$ as a $\Q[z]$-linear combination of the $\sigma(\alpha)$, $\sigma\in G$, yields an expression
\begin{equation} \label{eqcp1}
g_\isoul '(z) = \sum_{\sigma\in G} \sigma(\alpha)\sum_{\isoul'' \in\calI } R_{\isoul,\isoul'',\sigma}(z) g_{\isoul'',\Id}(z) 
\end{equation}
with $R_{\isoul,\isoul'',\sigma}(z) \in\Q(z) $. On the other hand, Eq.~\eqref{eqdcpgi} yields
\begin{equation} \label{eqcp2}
g_\isoul '(z) = \sum_{\sigma\in G} \sigma(\alpha) g_{\isoul ,\sigma}'(z). 
\end{equation}
Comparing the components on $\alpha$ of Eqs.~\eqref{eqcp1} and \eqref{eqcp2}, unicity of such an expression yields
$$g_{\isoul ,\Id}'(z) = \sum_{\isoul'' \in\calI }R_{\isoul,\isoul'',\Id}(z) g_{\isoul'',\Id}(z) . $$
This concludes the proof that the vector of $E$-functions $g_{ \isoul ,\Id}(z)$, $\isoul\in\calI$, satisfies a first-order linear differential system with coefficients in $\Q(z)$.

Now we come back to $\varpi$: combining Eqns.~\eqref{eqdefvp},~\eqref{eqdcpgi} and \eqref{eqegalg} yields
\begin{equation} \label{eqdcpsurQ}
\varpi = \sum_{\isoul\in\calI} \Big( \prod_{\sigma\in G} \sigma(\lambda_{i_\sigma})\Big)
\sum_{\tau\in G}\tau(\alpha) g_{\tau(\isoul), \Id}(1) = \sum_{\isoul'\in\calI} \kappa_{\isoul'} g_{\isoul', \Id}(1)
\end{equation}
upon letting 
$$\kappa_{\isoul'} = \sum_{\tau\in G}\tau(\alpha)\prod_{\sigma\in G} \sigma(\lambda_{i'_{\tau^{-1}\circ\sigma}})\in\K. $$

Let us prove that $\kappa_{\isoul}\in \Z$ for any $ \isoul\in\calI$. Indeed for any $\gamma\in G$, we have:
$$
\gamma(\kappa_\isoul) = \sum_{\tau\in G}\gamma(\tau(\alpha)) \prod_{\sigma\in G} \gamma(\sigma(\lambda_{i_{\tau^{-1}\circ\sigma}})) 
= \sum_{\tau'\in G}\tau'(\alpha)\prod_{\sigma'\in G} \sigma'(\lambda_{i_{\tau'^{-1}\circ\sigma'}}) = \kappa_{\isoul}
$$
by letting $ \tau' = \gamma\circ\tau$ and $\sigma'=\gamma\circ\sigma$, since $\tau'^{-1}\circ\sigma' = \tau^{-1}\circ\sigma$.
Using the fact that all $\tau(\alpha)$ and all $\lambda_j$ belong to $\OK$, we deduce that $\kappa_\isoul\in \OK\cap\Q=\Z$. 

To sum up, Eq.~\eqref{eqdcpsurQ} shows that $\varpi$ is a $\Z$-linear combination of the values at 1 of a family of cardinality $\Card ( \calI ) = N^d$  of $E$-functions with coefficients in $\Q$, solution of a first order differential system. Therefore Step 1 applies with $H' := H^d \sum_{\tau\in G}|\tau(\alpha)|$, since $|\kappa_\isoul|\leq H'$ for any $\isoul$. We obtain $| \varpi | > c H^{-dN^d+d-\eps}$ for any $\eps>0$, where $c>0$ depends on $\eps$. 
Now Eq.~\eqref{eqdefvarpi} yields $| \varpi | \leq c' H^{d-1} |\Lambda|$ by bounding trivially the factors corresponding to all $\sigma\neq\Id$; here $c'$ is a positive constant that depends only on $f_1$, \ldots, $f_N$ and $\K$. Combining these estimates yields $|\Lambda|> c'' H^{- dN^d +1-\eps } $ for some constant $c''$; this concludes Step 2 in the case where $\K$ is a Galois extension of $\Q$.

\medskip

If $\K/\Q$ (of degree $d$) is not assumed to be Galois, we consider a finite Galois extension $\L$ of $\Q$ such that $\K\subset \L$. We now explain the changes that must be made to the above construction. We let $G_0=\Gal(\L/\Q)$ and $H=\Gal(\L/\K)$. In the definition of $\varpi$, namely Eq.~\eqref{eqdefvarpi}, the product is now taken over the $d$ cosets $\sigma\in G_0/H$; indeed $\sigma(\lambda_j)$ and $f_j^\sigma$ are the same for all $\sigma$ in a given coset, because $\lambda_j$ and the coefficients of $f_j$ belong to $\K$.
In the products of Eqns.~\eqref{eqdefvp} and~\eqref{eqdefgi}, $\sigma$ ranges through $G_0/H$, and $\calI = \{1,\ldots,N\}^{ G_0/H }$. However the normal basis theorem is applied to the Galois extension $\L/\Q$, so that $\alpha\in\OL$ and $\sigma$ ranges through $G_0$ in Eqns.~\eqref{eqdcpgi} to~\eqref{eqcp2}. In Eq.~\eqref{eqdcpsurQ}, the product is over $\sigma\in G_0/H$ and the sum over $\tau\in G_0$. We have $\kappa_{\isoul}\in \L$, and deduce that $\kappa_{\isoul}\in \Q$ since $\gamma(\kappa_{\isoul}) = \kappa_{\isoul}$ for any $\gamma\in G_0$. We conclude the proof in the same way since $\Card(\calI)=N^d$.

\medskip

\noindent {\bf Alternative proof of Step 2 if $\lambda_j\in\Z$ for any $j$}. For the reader's convenience we give now a slightly different proof in this special case. It is based on the same idea of considering $\varpi$, but its expansion and the way Step 1 is applied are not the same. As in Step 2, we assume $\K/\Q$ to be Galois (the general case is dealt with as explained at the end of Step 2), let $G=\Gal(\K/\Q)$ and consider
\begin{equation} \label{eqdefvarpianc}
\varpi = \prod_{\sigma\in G} \Big( \sum_{j=1}^N \lambda_j f_j^\sigma(1)\Big)
\end{equation}
since we have $\sigma(\lambda_j)=\lambda_j$ now; we still have $\varpi\neq 0$. To expand the product in the definition of $\varpi$, we denote by $\calN$ the set of all tuples $\nsoul = (n_1,\ldots,n_N)$ of non-negative integers such that $ n_1 + \ldots + n_N = d$. For any $\nsoul\in\calN$, we denote by $I(\nsoul) $ the set of all tuples $\isoul = (i_\sigma)_{\sigma\in G}$ consisting of	 integers $i_\sigma\in\{1,\ldots,N\}$ such that for any $j\in\{1,\ldots,N\}$ we have:
$$\Card \{ \sigma\in G, \, \, i_\sigma=j\}=n_j.$$ 
Then Eq.~\eqref{eqdefvarpianc} yields
\begin{equation} \label{eqvarpi2}
\varpi = \sum_{\nsoul \in \calN} \lambda_1^{n_1} \cdot \ldots \cdot \lambda_N^{n_N} \varphi_{\nsoul}(1) \quad \mbox{ upon letting }\quad 
\varphi_{\nsoul}(z) := \sum_{ \isoul\in I(\nsoul)} \prod_{\sigma\in G} \ f_{i_\sigma}^\sigma(z) .
\end{equation}

Let us prove that $\varphi_{\nsoul}(z)$, which is an $E$-function with coefficients in $\K$, actually has coefficients in $\Q$ for any $\nsoul\in\calN$. For any $\tau\in G$ we have:
\begin{align*}
\varphi_{\nsoul} ^\tau 
&= \sum_{ \isoul\in I(\nsoul)} \prod_{\sigma\in G} \ \Big( f_{i_\sigma}^\sigma \Big)^\tau \\
&= \sum_{ \isoul\in I(\nsoul)} \prod_{\sigma\in G} \ f_{i_\sigma}^{\tau\circ\sigma} \\
&= \sum_{ \isoul\in I(\nsoul)} \prod_{\sigma'\in G} \ f_{i_{\tau^{-1}\circ\sigma'}}^{ \sigma'} \quad \mbox{ by letting } \sigma'=\tau\circ\sigma\\
&= \sum_{ \isoul' \in I(\nsoul)} \prod_{\sigma'\in G} \ f_{i'_{ \sigma'}}^{ \sigma'} 
\end{align*}
where the last equality comes from letting $ i'_\sigma = i_{\tau^{-1}\circ\sigma}$ for any $\sigma\in G$; indeed this defines a bijective map $I(\nsoul)\to I(\nsoul)$. Therefore $\varphi_{\nsoul} ^\tau = \varphi_{\nsoul}$ for any $\tau\in G$, and the $E$-function $\varphi_{\nsoul}(z)$ has coefficients in $\Q$. 

We denote by $\calE$ the vector space spanned over $\Qbar(z)$ by the functions 
$\prod_{\sigma\in G} \ f_{i_{ \sigma}}^{ \sigma} $ for all tuples $\isoul = (i_\sigma)_{\sigma\in G}$ consisting of integers $i_\sigma\in\{1,\ldots,N\}$. There are $N^d$ such tuples, so $\dim(\calE)\leq N^d$. Moreover we have $g'\in\calE $ for any $g\in\calE$. 

Let $\delta$ denote the dimension of the vector space spanned over $\Q(z)$ by the functions $\varphi_{\nsoul}$ for $\nsoul\in\calN$. We can choose $\delta$ functions $h_1, \ldots, h_\delta$ among the $\varphi_{\nsoul}$ which are linearly independent, and span the same $\Q(z)$-vector space. Choosing among the successive derivatives of $h_1, \ldots, h_\delta$ it is possible to find an integer $\delta'\geq \delta$ and functions $h_i$, for $\delta+1\leq i \leq \delta'$, such that $h_1, \ldots, h_{\delta'}$ are linearly independent over $\Q(z)$ and satisfy a linear differential system of order 1. Since they have rational coefficients, they are also linearly independent over $\Qbar(z)$; now they all belong to $\calE$, so we have $\delta'\leq \dim(\calE)\leq N^d$. 

Proposition~\ref{prop2} with $\mathbb K=\mathbb Q$ yields a vector of $E$-functions $e_1,\ldots,e_{\delta'}$ with rational coefficients, solution of a first-order differential system with no finite non-zero singularity, such that each $h_i$ is a linear combination of $e_1,\ldots,e_{\delta'}$ with coefficients in $\Q[z]$. There exist $R_{\nsoul,i}, S_{\nsoul,i}\in\Q(z)$ for $\nsoul\in\calN$ and $1\leq i \leq \delta'$ such that, for any $\nsoul$, 
$$ \varphi_{\nsoul}(z) = \sum_{i=1}^{\delta'} R_{\nsoul,i}(z) h_i(z) = \sum_{i=1}^{\delta'} S_{\nsoul,i}(z) e_i(z).$$
If no $ S_{\nsoul,i}$ has a pole at $z=1$, we can take $z=1$ in this equation. To deal with the general case, we expand the right hand side as a polynomial in $1/(z-1)$, up to an additive term which is holomorphic and vanishes at $z=1$. Since $ \varphi_{\nsoul}(z)$ is holomorphic at 1, all polar contributions cancel out and the value at $z=1$ is given by the constant term of the above-mentioned polynomial. This provides an expression of the form 
$$ \varphi_{\nsoul}(1) = \sum_{i=1}^{\delta'} \sum_{j=0}^{J} a_{\nsoul,i,j} e_i^{(j)}(1) $$
with $a_{\nsoul,i,j}\in \Q$. Since ${}^t(e_1,\ldots,e_{\delta'})$ is solution of a first-order differential system with coefficients in $\mathbb Q[z,1/z]$, hence with no finite non-zero singularity, we obtain finally
\begin{equation}\label{eq:2512}
 \varphi_{\nsoul}(1) = \sum_{i=1}^{\delta'} b_{\nsoul,i} e_i (1) 
\end{equation}
with $b_{\nsoul,i}\in \Q$ (where simply $b_{\nsoul,i}:=S_{\nsoul,i}(1)$ in the ``no pole at $z=1$'' case considered above).
Using Eq.~\eqref{eq:2512} into Eq.~\eqref{eqvarpi2} yields
$$
\varpi = \sum_{i=1}^{\delta'} \mu_i e_i(1)
\quad \mbox{ with }\quad
\mu_i = \sum_{\nsoul \in \calN} \lambda_1^{n_1} \cdot \ldots \cdot \lambda_N^{n_N} b_{\nsoul,i}\in\Q.
$$

This enables us to apply the special case of Theorem~\ref{th1} where $\K=\Q$, proved in Step 1, with $N$ replaced with $\delta'\leq N^d$. 
Indeed we denote by $\alpha\in\Z$ a common positive denominator of the rational numbers $ b_{\nsoul,i} $; then we have
$\alpha\mu_1,\ldots, \alpha\mu_{\delta'}\in\Z$.
Since $\varpi \neq 0$ we obtain $|\alpha \varpi | > c {H'}^{-N^d+1-\eps}$ where 
$$
H'=\max_{1\leq i \leq \delta} |\alpha \mu_i| \leq \beta 
\max_{\nsoul \in \calN} \house{ \lambda_1^{n_1} \cdot \ldots \cdot \lambda_N^{n_N}} \leq \beta H^d
$$ 
where $\beta>0$ depends only on $f_1,\ldots,f_N$ and $\K$. Then we conclude the proof as in Step 2.

\section{Decomposition of $E$-functions over a number field} \label{secdcp}

In the same spirit as Proposition~\ref{prop1}, it is possible to prove the following result. The weaker version with $\K$ replaced by $\Qbar$ was first proved in the unpublished note \cite{ri2016}, and the special case $\K=\mathbb Q$ in \cite{BRS2}.

\begin{prop} \label{prop3}
Let $f$ be an $E$-function with coefficients in a number field $\K$. Then there exist polynomials $P,Q\in\K[z]$, and an $E$-function $g$ with coefficients in $\K$, such that
$$f(z) = P(z) + Q(z)g(z) \mbox{ and $g(z_0)$ is transcendental for all $z_0\in\Qbar^\ast$.}$$
\end{prop} 

In this setting, the non-zero algebraic numbers $z$ at which a transcendental $f$ takes an algebraic value are exactly the roots of $Q$. Moreover, replacing $P$ with its remainder in its Euclidean division by $Q$, we may assume $\deg P < \deg Q$ provided $Q\neq 0$ (i.e., when $f$ is not a polynomial) and unicity then holds if $Q$ is monic such that $Q(0)\neq 0$ (properties which can both be assumed without loss of generality); see \cite[Proposition 3.3]{BRS2}. 

\medskip

Proposition~\ref{prop3} is a generalization of the following result, which will be used in the proof. It is stated as \cite[Theorem 4]{ateo} and its proof is due to the referee of \cite{gvalues}.

\begin{lem} \label{lemateo}
Let $f$ be an $E$-function with coefficients in a number field $\K$, and $\alpha\in\Qbar$ be such that $f(\alpha)$ is algebraic. Then $f(\alpha)\in \K(\alpha)$.
\end{lem}

This lemma asserts that $\EE_{\K(\alpha)}\cap\Qbar=\K(\alpha)$; it is a consequence of Theorem~\ref{theotens}.

For the convenience of the reader, let us deduce Lemma~\ref{lemateo} from Proposition~\ref{prop1}. Let $\beta = f(\alpha)$, and $\L$ be a finite Galois extension of $\K(\alpha)$ such that $\beta\in\L$. Since $f(z)-\beta$ vanishes at $\alpha$, Proposition~\ref{prop1} shows that for any $\sigma\in \Gal(\L/\K(\alpha))$ 
the $E$-function $f(z)-\sigma(\beta) = f^\sigma(z)-\sigma(\beta) $ vanishes at $\sigma(\alpha)=\alpha$, so that $ \sigma(\beta) =f(\alpha) = \beta$. 
This concludes the proof of Lemma~\ref{lemateo}.

\medskip

\begin{proof}[Proof of Proposition~\ref{prop3}] To prove Proposition~\ref{prop3}, we first remark that the result is obvious if $f$ is algebraic, hence a polynomial: we simply take $P=f$ and $Q=0$. Let us now assume that $f$ is transcendental and  argue by induction on the number of non-zero algebraic numbers  $\alpha$ such that $f(\alpha)\in \Qbar$. By Beukers' theorem  this number is finite, indeed, since $f$ is transcendental any such 
  $\alpha$ must be one of the finitely many singularities of an appropriate differential equation of which $f$ is a solution.

If this number is 0, one may choose $P=0$ and $Q=1$. Now if $f(\alpha)\in \Qbar$, Lemma~\ref{lemateo} proves that $ f(\alpha)$ belongs to $\K(\alpha) $: there exists $P_0\in\K[X]$ such that $ f(\alpha) = P_0(\alpha)$. Therefore the $E$-function $f-P_0$, with coefficients in $\K$, vanishes at $\alpha$. Proposition~\ref{prop1} yields an $E$-function $g_0$ with coefficients in $\K$ such that $f = P_0 + D g_0$ where $D $ is the minimal polynomial of $\alpha$ over $\K$. If $g_0(\alpha)\in\Qbar$, the same procedure can be carried out with $g_0$, leading to $P_1\in\K[z]$ and an $E$-function $g_1$ with coefficients in $\K$ such that $g_0=P_1+Dg_1$. After finitely many steps, this procedure terminates and provides $g_\ell$ such that $g_\ell(\alpha)\not\in\Qbar$ (see the proof of \cite[Theorem 3.4]{BRS2}). This concludes the proof of Proposition~\ref{prop3}. \end{proof}

\section{Structure of $\EK$} \label{sec5}

 In this section we prove Theorem~\ref{theotens} stated in the introduction.
Let $\K$ be a subfield of $\Qbar$. Recall that $\EK$ is the ring of all values $f(1)$ where $f$ is an $E$-function with coefficients in $\K$; in particular $\EQbar = \EE$. 

\medskip

Let $f_1$, \ldots, $f_N$ be $E$-functions with coefficients in $\K$. If $f_1(1)$, \ldots, $f_N(1)$ are linearly independent over $\Qbar$, then obviously they are linearly independent over $\K$. Conversely, let us assume that they are linearly independent over $\K$. Let $\lambda_1$, \ldots, $\lambda_N$ be algebraic numbers, not all zero, such that $ \lambda_1 f_1(1) + \cdots + \lambda_N f_N(1)=0$. Up to a permutation of the indices we may assume that $\lambda_1\neq 0$; then dividing by $\lambda_1$ we assume that $\lambda_1=1$. Let us consider a finite Galois extension $\L$ of $\K$ that contains 
$\lambda_2$, \ldots, $\lambda_N$. Then $g(z) = \sum_{i=1}^N \lambda_i f_i(z)$ is an $E$-function with coefficients in $\L$, and it vanishes at $z=1$. For any $\sigma\in\Gal(\L/\K)$, Proposition~\ref{prop1} yields $g^\sigma(1)=0$, that is $\sum_{i=1}^N \sigma(\lambda_i) f_i(1)=0$ since all $f_i$ have coefficients in $\K$. Summing these relations, as $\sigma$ varies, yields
$$\sum_{i=1}^N {\rm Tr}_{\L/\K} (\lambda_i) f_i(1)=0$$
with $ {\rm Tr}_{\L/\K} (\lambda_i) = \sum_{\sigma\in\Gal(\L/\K)} \sigma(\lambda_i) \in \K$ and $ {\rm Tr}_{\L/\K} (\lambda_1 ) = {\rm Tr}_{\L/\K} (1) = [\L:\K]\neq 0$. This is a non-trivial linear relation, with coefficients in $\K$, between $f_1(1)$, \ldots, $f_N(1)$. This contradiction concludes the proof that elements of $\EK$ are linearly independent over $\Qbar$ if, and only if, they are linearly independent over $\K$.

\section{A Galois action on values of $E$-functions} \label{secGalois}

In this section, we define an action of $\Gal(\Qbar/\Q)$ on the set $\EE$ of values of $E$-functions. 
This action is not used in the paper but it sheds a different light on the proof of Theorem~\ref{th1}: it presents similarities with Liouville's proof that irrational algebraic numbers are not too well approximated by rationals (i.e., are not Liouville numbers).

 Given $\sigma\in \Gal(\Qbar/\Q)$ and $\xi\in\EE$, there exists an $E$-function $f$ such that $\xi=f(1)$; then we let $\sigma(\xi) := f^\sigma(1)$. The crucial point is to prove that $\sigma(\xi) $ depends only on $\sigma $ and $\xi$, not on the choice of $f$. Indeed if $g$ is another $E$-function such that $\xi=g(1)$, then $f-g$ vanishes at the point 1. Proposition~\ref{prop1} shows that $f^\sigma-g^\sigma$ vanishes at 1 too, so that $ g^\sigma(1)= f^\sigma(1)$: this concludes the proof.

Theorem \ref{theotens} shows that as a Galois representation, $\EE$ is isomorphic to $\EQ\otimes \Qbar$ where $\EQ$ is a $\Q$-vector space with trivial Galois action. Therefore this does not provide a way to understand the absolute Galois group of $\Q$ better. However, it  sheds a new light on the proof of Theorem   \ref{th1} (see \S \ref{secpreuveth1}): Step 2  is very similar to Liouville's proof that irrational algebraic numbers are not too well approximated by rationals (i.e., are not Liouville numbers).
Indeed let us recall briefly Liouville's proof, stated in terms of Galois action. Let $\xi$ be an algebraic number of degree $d\geq 2$, and assume (for simplicity) that the extension $\Q(\xi)/\Q$ is Galois. To bound from below $| q\xi-p|$ for $(p,q)\in\Z^2\setminus\{(0,0)\}$, consider 
$$\varpi := \prod_{\sigma\in \Gal(\Q(\xi)/\Q)} \big( q \sigma(\xi)-p\big).$$
Then $\varpi\neq 0$ (since $\sigma(\xi)$ is irrational for any $\sigma$), and $\varpi\in\Q$ (since it is the norm of $ q\xi-p $ with respect to the extension $\Q(\xi)/\Q$). Letting $\delta\in\Z$ denote a positive integer such that $\delta \xi$ is an algebraic integer, we have $\delta^d \varpi\in\Z\setminus\{0\}$ since $\OQdexi\cap\Q=\Z$. Therefore $|\delta^d\varpi |\geq 1 $ so that 
$$\delta^{-d} \leq | \varpi | \leq | q\xi-p| (\house{\xi}+1)^{d-1} H^{d-1}$$
by bounding $| q \sigma(\xi)-p|$ trivially for $\sigma\neq \Id$, where $H = \max(|p|,|q|)$. Dividing by $q$ yields $| \xi-p/q|\geq c H^{-d}$ where $c>0$ depends only on $\xi$.

Step 2 of the proof of Theorem \ref{th1} presents similarities, except that elements of $\K\subset \Qbar$ are replaced with values of $E$-functions in $\EK$; the lower bound used by Liouville (namely, $\delta^d \varpi\in\Z\setminus\{0\}$ implies $|\delta^d \varpi|\geq 1 $) is replaced accordingly by Shidlovskii's lower bound recalled in Theorem~\ref{thshid}.

\medskip

\noindent St\'ephane Fischler, Universit\'e Paris-Saclay, CNRS, Laboratoire de math\'ema\-ti\-ques d'Orsay, 91405 Orsay, France.

\medskip

\noindent Tanguy Rivoal, Universit\'e Grenoble Alpes, CNRS, Institut Fourier, CS \!40700, 38058 Grenoble cedex 9, France.

\bigskip

\noindent Keywords: $E$-functions, Andr\'e-Beukers Theorems, Linear independence me\-asures, Irrationality measures,Transcendence measures, Liouville numbers, Sh\-id\-lov\-skii's Theorem.
\medskip

\noindent MSC 2020: 11J82 (Primary), 11J91 (Secondary)

\begin{thebibliography}{10}

\bibitem{ac} B. Adamczewski, J. Cassaigne, 
Diophantine properties of real numbers generated by finite automata, 
{\em Compos. Math.} {\bf 142}.6 (2006), 1351--1372.

\bibitem{AF} B. Adamczewski, C. Faverjon, M\'ethode de 
Mahler : relations lin\'eaires, transcendance et
applications aux nombres automatiques, 
{\em Proc. London Math. Soc.} {\bf 115}.3 (2017), 55--90.

\bibitem{andrelivre} Y. André, {\em G-functions and Geometry}, Aspects of Mathematics {\bf E13}, Friedr. Vieweg \& Sohn, Braunschweig, 1989.

\bibitem{andre} Y. Andr\'e, S\'eries Gevrey de type arithm\'etique I. Th\'eor\`emes de puret\'e et de dualit\'e,
{\em Annals of Math.} {\bf 151}.2 (2000), 705--740.

\bibitem{andreasens} Y. Andr\'e, Solution algebras of differential equations and quasi-homogeneous varieties: a new differential Galois correspondence, {\em Ann. Sci. \'Ec. Norm. Sup\'er.} {\bf 47}.2 (2014), 449--467.

\bibitem{baker} A. Baker, {\em Transcendental Number Theory}, Cambridge University Press, second edition, 1990.

\bibitem{bbc} J. P. Bell, Y. Bugeaud,  M. Coons, Diophantine approximation of Mahler numbers, {\em Proc. Lond. Math. Soc.} (3) {\bf 110}.5 (2015), 1157--1206.

\bibitem{bb} D. Bertrand, F. Beukers, \'Equations diff\'erentielles lin\'eaires et majorations de multiplicit\'es, {\em Ann. Sci. \'Ec. Norm. Sup\'er.} (4) {\bf 18}.1 (1985),  181--192.

\bibitem{bcy} D. Bertrand, V. Chirskii, J. Yebbou, Effective estimates for global relations on Euler-type series, {\em Ann. Fac. Sci. Toulouse Math.} (6) {\bf 13}.2 (2004), 241--260.


\bibitem{beukers} F.~Beukers, {A refined version of the Siegel-Shidlovskii theorem}, {\em Annals of Math.} {\bf 163} (2006), 369--379. 


\bibitem{BRS2} A. Bostan, T. Rivoal, B. Salvy, {Minimization of differential equations and algebraic values of $E$-functions}, preprint (2022), 37 pages.

\bibitem{Bourbaki} N. Bourbaki, {\em Elements of Mathematics, Algebra II}, Chapters 4--7, Springer, 2003.

\bibitem{bugeaud} Y. Bugeaud, {\em Approximation by Algebraic Numbers}, Cambridge Tracts in Mathematics {\bf 160}, 2004.

\bibitem{EMS} N. I. Feldman, Yu. V. Nesterenko, {Transcendental Numbers}, in {\em Encyclopaedia of Mathematical Sciences}, Vol. {\bf 44}: Number Theory IV (Springer, 1998).

\bibitem{gvalues} S. Fischler, T. Rivoal, On the values of $G$-functions, {\em Commentarii Math. Helv.} {\bf 89}.2 (2014), 313--341.

\bibitem{ateo} S. Fischler, T. Rivoal, {Arithmetic theory of $E$-operators}, {\em Journal de l'\'Ecole polytechnique -- Math\'ematiques} {\bf 3} (2016), 31--65. 

\bibitem{fuchsien} S. Fischler, T. Rivoal, {Microsolutions of differential operators and values of arithmetic Gevrey series}, {\em American J. of Math.} {\bf 140}.2 (2018), 317--348.

\bibitem{kappe} L.-C. Kappe, Zur Approximation von $e^\alpha$, {\em Ann. Univ. Sci. Budap. Rolando Eötvös, Sect. Math.} {\bf 9} (1966), 3--14.

\bibitem{lepetit} G. Lepetit, Le th\'eor\`eme d'Andr\'e-Chudnovsky-Katz au sens large, {\em North-West. Eur. J. Math.} {\bf 7} (2021), 83--149. 


\bibitem{mahler} K. Mahler, Zur Approximation der Exponentialfunktion und des Logarithmus. Teil I, {\em J. reine angew. Math.} {\bf 166} (1931), 118--136.


\bibitem{Milne} J. Milne, {\em Fields and Galois theory}, version 5.10, available at {\tt www.jmilne.org/math/}, 144 pages, 2022.


\bibitem{ns} Yu. V. Nesterenko, A. B. Shidlovskii, On the linear independence of values of $E$-functions, {\em Sb. Math.} {\bf 187} (1996), 1197--1211; translation from the russian {\em Math. Sb.} {\bf 187} (1996), 93--108.


\bibitem{ri2016} T. Rivoal, Valeurs alg\'ebriques de $E$-fonctions aux points alg\'ebriques, unpublished note (2016), 4 pages, available at {\tt https://hal.archives-ouvertes.fr/hal-03676576}

\bibitem{Shidlovskiilivre} A.~B. Shidlovskii, {\em Transcendental numbers}, de Gruyter Studies in Math., no.~12, de Gruyter, Berlin, 1989.

\bibitem{siegel} C. Siegel, \"Uber einige Anwendungen diophantischer Approximationen, vol. 1 S. {\em Abhandlungen Akad.}, Berlin, 1929.

\bibitem{zudilin} W. Zudilin, On rational approximations of values of a certain class of entire functions, 
{\em Sb. Math.} {\bf 186}.4 (1995), 555--590; translation from the russian {\em Mat. Sb.} {\bf 186}.4 (1995), 89--124. 

\end{thebibliography}
\end{document}